\documentclass[12pt]{amsart}
\usepackage[english]{babel}
\usepackage{amsmath,amsthm}
\usepackage{amsfonts,comment}
\usepackage{comment}
\usepackage{graphicx}
\usepackage{amscd,color}

\newtheorem{thm}{Theorem}[section]

\newtheorem{lemma}[thm]{Lemma}

\newtheorem{rem}{Remark}[section]

\def\B{\mathcal B}

\def\Bn{L_{B^{1/n}}}
\def\B2{L_{B^{1/2}}}
\def\Bs{L_{B^{s/n}}}
\def\Bs+{L_{B^{(s+1)/n}}}

\def\Xn{L_{B^{1/n}}}
\def\X2{L_{B^{1/2}}}
\def\Xs{L_{B^{s/n}}}
\def\Xs+{L_{B^{(s+1)/n}}}

\def\Xn{L_{X^{1/n}}}
\def\X2{L_{X^{1/2}}}
\def\Xs{L_{X^{s/n}}}
\def\Xs+{L_{X^{(s+1)/n}}}

\def\Xn{L_{X^{1/n}}}
\def\X2{L_{X^{1/2}}}
\def\Xs{L_{X^{s/n}}}
\def\Xs+{L_{X^{(s+1)/n}}}

\def\CC{\mathbb C}

\def\HH{\mathbb H}

\def\RR{\mathbb R}

\numberwithin{equation}{section}
\begin{document}

\title{Adjoining Roots and Rational Powers  of Generators in $PSL(2,\RR)$  and Discreteness}
\author{Jane Gilman}

\address{Mathematics and Computer Science Department, Rutgers University, Newark, NJ 07102}

\email{gilman@rutgers.edu, jgilman@math.princeton.edu}%

\thanks{Some of this work was carried out while the author was partially supported as Visiting Research Scholar by Princeton University}



\date{revision of \today}


\maketitle

\begin{abstract} Let $G$ be a finitely generated group of isometries of $\HH^m$, hyperbolic $m$-space,  for some positive integer $m$. 
  The discreteness problem is to determine whether or not $G$ is discrete. Even in the case of a  two generator non-elementary subgroup of $\HH^2$ (equivalently $PSL(2,\mathbb{R})$) the   problem requires an algorithm \cite{GM,JGtwo}. If $G$ is discrete, one can ask when adjoining an $n$th root of a generator results in a discrete group.

   In this paper we address the issue for pairs of hyperbolic generators in $PSL(2, \RR)$ with disjoint axes and obtain necessary and sufficient conditions for adjoining roots for the case when the two hyperbolics have a hyperbolic product and are what as known as {\sl stopping generators} for the Gilman-Maskit algorithm \cite{GM}. We give an algorithmic solution in other cases. It applies to all other types of pair of generators that arise in what is known as the {\sl intertwining case}. The results are geometrically motivated and stated as such, but also can be given computationally using the corresponding matrices.
\end{abstract}

 \noindent Assume that $B$ is hyperbolic or parabolic, then  $\hat{G}= \langle A, B^{1/2} \rangle$ is discrete and free $\iff$ either

     \noindent  but  neither intersection is a vertex of the hexagon interior to $\HH^2$.
\vskip .1in

\noindent If $A^{-1}B$ is primitive elliptic,  then  $\hat{G}= \langle A, B^{1/2}  \rangle$ is discrete  $\iff$
 either

 (i) $L_{B^{1/2}} \cap Ax_A \ne \emptyset$ or
 (ii) $L_{B^{1/2}} \cap Ax_{A^{-1}B}  \ne \emptyset$ or
 (iii) $L_{B^{1/2}} \cap L_A \ne \emptyset$ and $A^{-1}B$ is primitive.

 \noindent but neither intersection is a vertex of the hexagon interior to $\HH^2$.
\noindent {\bf (P)} If $A^{-1}B$ is parabolic, then  $\hat{G}= \langle A, B^{1/2} \rangle$ is discrete and free $\iff$
 either

 (i) $L_{B^{1/2}} \cap Ax_A \ne \emptyset$ or
 (ii) $L_{B^{1/2}} \cap Ax_{A^{-1}B}  \ne \emptyset$

 \noindent but neither intersection is a vertex of the hexagon interior to $\HH^2$.
\vskip .1in
\noindent  $\hat{G}= \langle A, B^{1/2}  \rangle$ is discrete  $\iff$
 either

 (i) $L_{B^{1/2}} \cap Ax_A \ne \emptyset$ or

 (ii) $L_{B^{1/2}} \cap Ax_{A^{-1}B}  \ne \emptyset $

 \noindent but neither intersection is a vertex of the hexagon interior to $\HH^2$.

\vskip .15in
In all other cases, the group is not free and one applies the algorithm to the case where $A^{-1}B$ is elliptic to determine discreteness.
\end{thm}
\end{comment}

\section{Introduction}

Let $G$ be a finitely generated group of isometries of $\HH^m$,  hyperbolic $m$-space.
  The discreteness problem is to determine whether or not $G$ is discrete. Even the  two generator non-elementary discreteness problem in  $\HH^2$ (or equivalently $PSL(2,\mathbb{R})$)  requires an algorithm. One such algorithm is the Gilman-Maskit algorithm \cite{GM}, termed the GM algorithm for short and also known as the intertwining algorithm, taken together with the intersecting axes algorithm \cite{JGtwo}. The GM algorithm proceeds by considering  geometric types of the pairs of generators. If $G$ is discrete, one can ask when adjoining an $n$th root of a generator results in a discrete group.

Here we answer the discreteness question for adjoining for roots and rational powers of one or both  generators in non-elementary two generator discrete subgroups of $PSL(2,\RR)$ found by the GM algorithm, the intertwining algorithm.
When the algorithm is applied to a two generator group, the pair of generators at which discreteness is determined are termed {\sl discrete stopping generators} and the generators correspond to a certain geometric configuration which we review below (see Section \ref{sec:STOPPER})
. The main result of the this paper is discreteness conditions on adjoining roots of a discrete stopping generator. 
In all other cases, that is the cases of non-stopping generators, discreteness can be determined by running the algorithm using the root as one of the generators.

The problem of adjoining roots has been addressed in \cite{Beard, GCJ, Parker}. Beardon gave a necessary and sufficient condition for a discrete group generated by a pair of parabolics and  Parker obtained results for rational powers of a pair of generators in the case where neither generator was hyperbolic. In \cite{GCJ} discreteness conditions for hyperbolics are given by inequalities that depend upon the cross ratio and multipliers.
    Since our technique also applies to some of the intertwining cases that Beardon and Parker addressed but are different than their techniques, we include those cases, too.

The organization of this paper is as follows:
In sections
 \ref{sec:not} and \ref{sec:preliminaries} notation is fixed and prior results needed are summarized. Results for square roots, arbitrary roots and their powers appear as  Theorems \ref{thm:neccsuffSQRT}, \ref{thm:nthroots}, \ref{thm:ratls/n} and  \ref{thm:BPE}. Their proofs are given in sections \ref{sec:SQroots},\ref{sec:nth}, and \ref{sec:PE}. For example, in section \ref{sec:SQroots} we find necessary and sufficient conditions (Theorem \ref{thm:neccsuffSQRT}) for a group generated  by a pair of hyperbolics discrete stopping generators with hyperbolic product to be discrete and free when a square root of a stopping generators is added. The results are  extended to rational powers (Theorems \ref{thm:nthroots} and \ref{thm:ratls/n}) in Sections \ref{sec:nth} and
 \ref{sec:PE}. In section \ref{sec:allmainTHMS} these theorems are extended and stated in greater generality as Theorems \ref{thm:main}, \ref{thm:ratlpowersX}  and \ref{thm:moregen}.

\section{Preliminaries: Notation and Terminology \label{sec:not}}

We recall that elements of $Isom(\HH^2)$ and $Isom(\HH^3)$
are classified by their geometric action or equivalently by their traces when considered as elements of $PSL(2,\RR)$
or $PSL(2,\CC)$.
 In $\HH^2$ they are either hyperbolic, parabolic or elliptic and we use $H$,$P$ and $E$ to denote such an element type. We consider their action using the unit disc model for $\HH^2$.  A hyperbolic elements fixes two points on the boundary of the unit disc, its ends,  and the geodesic interior connecting these two points, its axis.  A parabolic fixes one point on the boundary of the unit disc and an elliptic fixes one point interior to the disc. In $\HH^3$ an elliptic element has an axis; in $\HH^2$ it is customary to consider the fixed point of an elliptic, its axis and in both $\HH^3$ and $\HH^2$ to consider the fixed point of the parabolic on the boundary of hyperbolic space its axis. All transformations fix their axes.

For any pair of points $r$ and $s$ in $\overline{\HH}^2$, we let $[r,s]$ denote the unique geodesic connecting the points. If $r$ is on the boundary we consider the point $r$ to be an (improper) geodesic following \cite{Fench},  denote it by  $[r,r]$. 

A hyperbolic transformation moves points along its axis a fixed distance in the hyperbolic metric, called it {\sl translation length} toward one end, the attracting fixed point on the boundary  and away from the other, the repelling fixed point. An elliptic transformation rotates by an angle $\theta$ about its fixed point where $\theta/2$ is the angle between the two geodesics $L$ and$M$ meeting at the fixed where the elliptic is the product of reflections in $L$ and $M$.

If $X$ is any geodesic, there is an orientation reversing  element of order two that fixes $X$ and its ends that is called {\sl the half-turn about $X$} and denoted by $H_X$. It is a reflection through $X$\footnote{In $\HH^3$ a half-turn about a geodesic is the orientation preserving element of order two fixing the geodesic point-wise. This can be viewed as the product of a reflection through the geodesic in any hyperbolic plane containing the geodesic and a reflection in the plane itself. Since the restriction of a half-turn to $\HH^2$ is a reflection, it is customary to use $H_M$ there instead of $R_M$. }.

Any hyperbolic element of $Isom(\HH)^2$  can be factored in many was as the product of two half turns about geodesics perpendicular to its axes. Here the two half-turn geodesics intersect the axis  half the translation length apart. An elliptic element it is the product of two half-turn geodesics  intersecting the axis (a point) and making an angle of $\theta/2$ with each other there. For a parabolic the half-turns geodesics intersect at the fixed point on the boundary.

The discreteness algorithm consists of two independent parts: {\sl the intertwining algorithm} \cite{GM}  addresses  pairs of hyperbolics with disjoint axes and other types of pairs that follow in this case and the intersecting axes case \cite{JGtwo}. The case of hyperbolics with intersecting axes and those that follow from it will be treated elsewhere. The ideas are similar but requires additional and different notation.

Note that we let $Ax_X$ denote the axis of $X$ if $X$ is any transformation but for clarity we sometimes also write for the axes (i) if $X$ is parabolic, the point on the boundary of the unit disc,   $pp_X$ or $[p_X,p_X]$ using notation for an improper line as in Fenchel \cite{Fench} and (ii) if $X$ is elliptic $p_X$ with fixed point interior to the unit disc.

Following \cite{Parker}, we note that an {\sl $n$th root of a hyperbolic or parabolic} (and thus any rational power) is  defined unambiguously. To define an {\sl $n$th root of an elliptic} we need to consider that it is always conjugate to $z \mapsto Kz$ considered as an isometry in $PSL(2,\CC)$ and to take the root there and then conjugate back. Further, a {\sl geometrically primitive} root of an elliptic is an element that corresponds to a minimal rotation in the cyclic group it generates. Thus if $E$ is a primitive rotation so is $E^{-1}$. An element is {\sl algebraically primitive} if it generates the entire cyclic group, but here we do not consider such elements to be primitive. The algorithm assumes one can determine whether or not an  elliptic is of finite order.

The figures here are schematic. All geodesics are perpendicular to the boundary the unit disc. Blue circles are used to indicate intersections that are perpendicular. Figures for some representative cases are presented, but these are not exhaustive.


\section{Preliminaries: the  GM algorithm, $G$, and Hexagons} \label{sec:preliminaries}

Assume that $G=\langle A, B \rangle$ is a non-elementary two generator subgroup of $PSL(2,\RR)$.
The Gilman-Maskit discreteness algorithm considers the intertwining cases. 
The GM algorithm  begins with a pair of hyperbolic generators
with disjoint axes and at each step either stop and outputs that the group is discrete or  that the group is not discrete, or outputs the {\sl next} pair of generators to consider.
An implementation of the algorithm can begin with any geometric type of pairs of generators that arise in the algorithm.
The generators where the algorithm outputs discreteness are termed the {\sl discrete stopping generators}

Given $A$ and $B$, there will be a unique  geodesic $L$, the {\sl core geodesic},  that is a common perpendicular to their axes. We assume that $L$ is oriented from the axis of $A$ towards the axis of $B$.

 Further, given $L$, we can find geodesics $L_A$ and $L_B$ such that $A= H_LH_{L_A}$ and $B = H_LH_{L_B}$ so that $A^{-1}B = H_{L_A}H_{L_B}$.
We let $3G = \langle H_L, H_{L_a}, H_{L_B} \rangle$. We note that $3G$ and $G$ are simultaneously discrete or non-discrete as $G$ is a subgroup of index $2$ in $3G$.

The axes and half-turn lines determine a geometric configuration, a hexagon (see \cite{Fench}). For a given $3G$ the  hexagon may or may not be convex. (See Figure \ref{fig:Convex-Non-convex} for examples of a convex and a non-convex hexagon.)  The hexagon will have three axis sides and three half-turn sides. In $\HH^2$ one or more of the axis sides may reduce to a point that is interior or on the boundary,  but the half-turn sides will not.

The geodesics that determine the sides of the hexagon will have subintervals that actually correspond to sides of the hexagon and it will generally be clear from the context when whether we are talking about a side or the entire geodesic.
For any positive integer $n$, there are geodesics $L_{B^2}, L_{B^3}, L_{B^4} \cdots $  such that $B = L_{B^{n-1}}L_{B^n}$ and $B^n =H_LH_{L_{B^n}}$. We can also find geodesics $\B2$ or $\Bn$ with $B^{1/2}= H_L H_{{\B2}}$ and
$B^{1/n}= H_L H_{\Bn}$. When $B$ is hyperbolic these geodesics are perpendicular to $Ax_B$. When $B$ is parabolic or elliptic, the geodesics pass trough the point that is $Ax_B$.
\begin{figure}
\begin{center}
\includegraphics[height=2in]{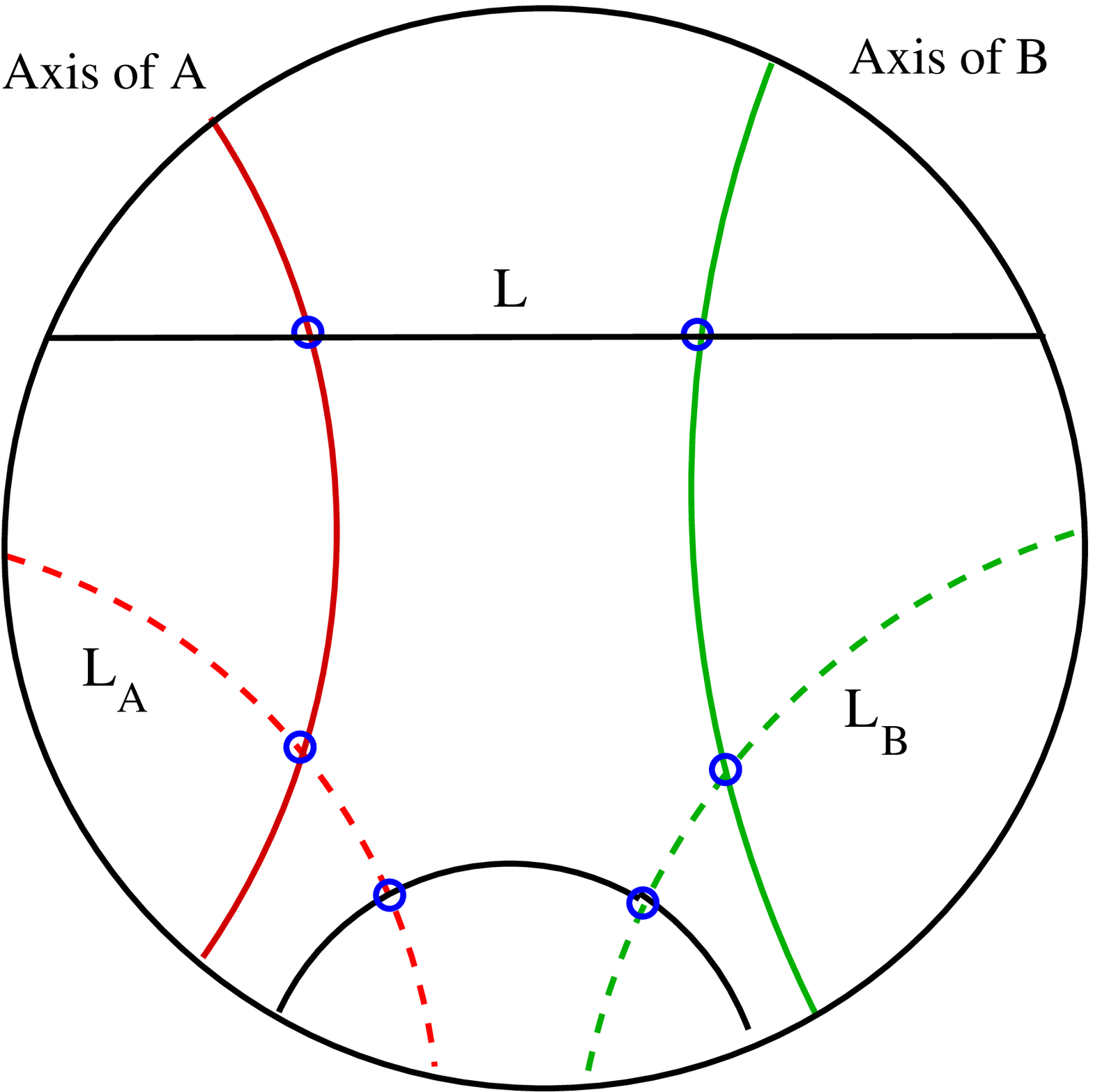}

\vskip .05in
$\;\;\;\;$ \includegraphics[height=2in ]{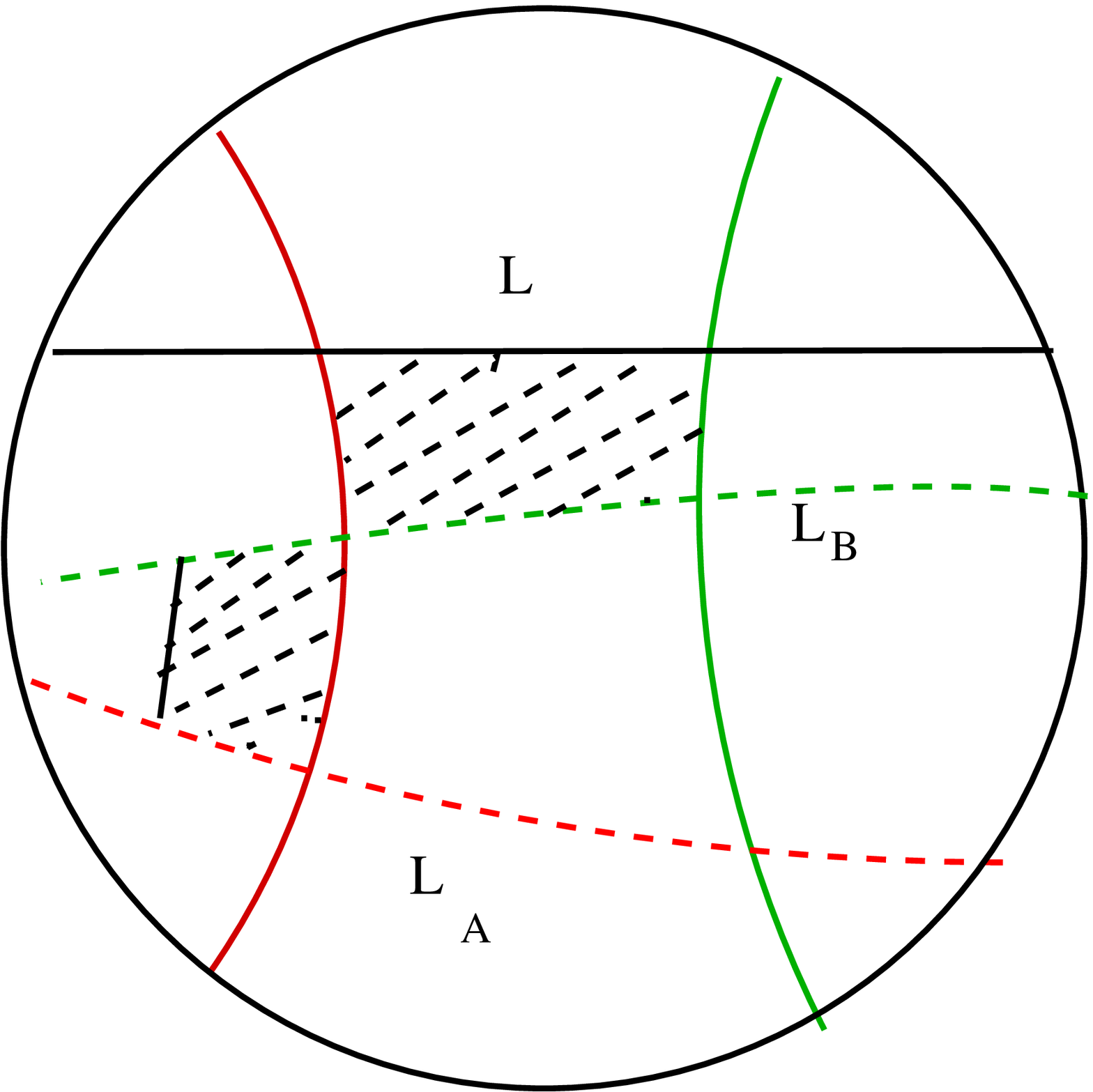}
\caption{A convex hexagon  and a self-intersecting, nonconvex hexagon \label{fig:Convex-Non-convex}}
\end{center}
\end{figure}

\section{First Result} \label{sec:first}
We begin with the following lemma.
\begin{lemma} Let $\mathcal{G}$ be any group generated by half-turns about three disjoint geodesics, $L,M,N$. If the half turn geodesics  bound a region, that is, no one half-turn geodesic  separates the other two, then $\mathcal{G}$ is discrete and free.

If one or more pairs of half-turns intersect, the product of the pair is elliptic or parabolic. If the product is
 elliptic of finite order and the angle between the half-turn geodesics of the elliptic is half a primitive angle or parabolic, then the group is discrete providing the half-turn geodesics still bound a region and the transformations are oriented so that the vertex angle hypotheses of the Poincar{\'e}{\'e} Polygon Theorem apply. If the pairs of half-turns only intersect on the boundary, then the group is also free. \end{lemma}

\begin{proof}

Apply the Poincar{\'e}{\'e} Polygon Theorem \cite{Beard} or \cite{Maskit}.

\end{proof}

\section{Geometric Stopping Generators and Discrete Stopping Configurations \label{sec:STOPPER}} 

 Let $H$,$P$ and $E$ denote respectively  a hyperbolic, parabolic, or elliptic generator.


Note that the stopping configurations which are all hexagons, may look like hyperbolic  pentagons, quadrilaterals or triangles because the axes may be  points in $\overline{\HH^2}$ and also note that  when we discuss discrete stopping generators that include elliptic we assume the generator to be geometrically primitive.

We illustrate some, but not all,  figures for the discrete stopping cases in Figure \ref{fig:someSTOP}.
\begin{thm} For each of the eleven possible ordered stopping configurations the hexagon is convex and satisfies the vertex hypotheses of the Poincar{\'e} Polygon theorem in the case of an elliptic generator.
\end{thm}

\begin{figure}\begin{center}
 \includegraphics[height=2in]{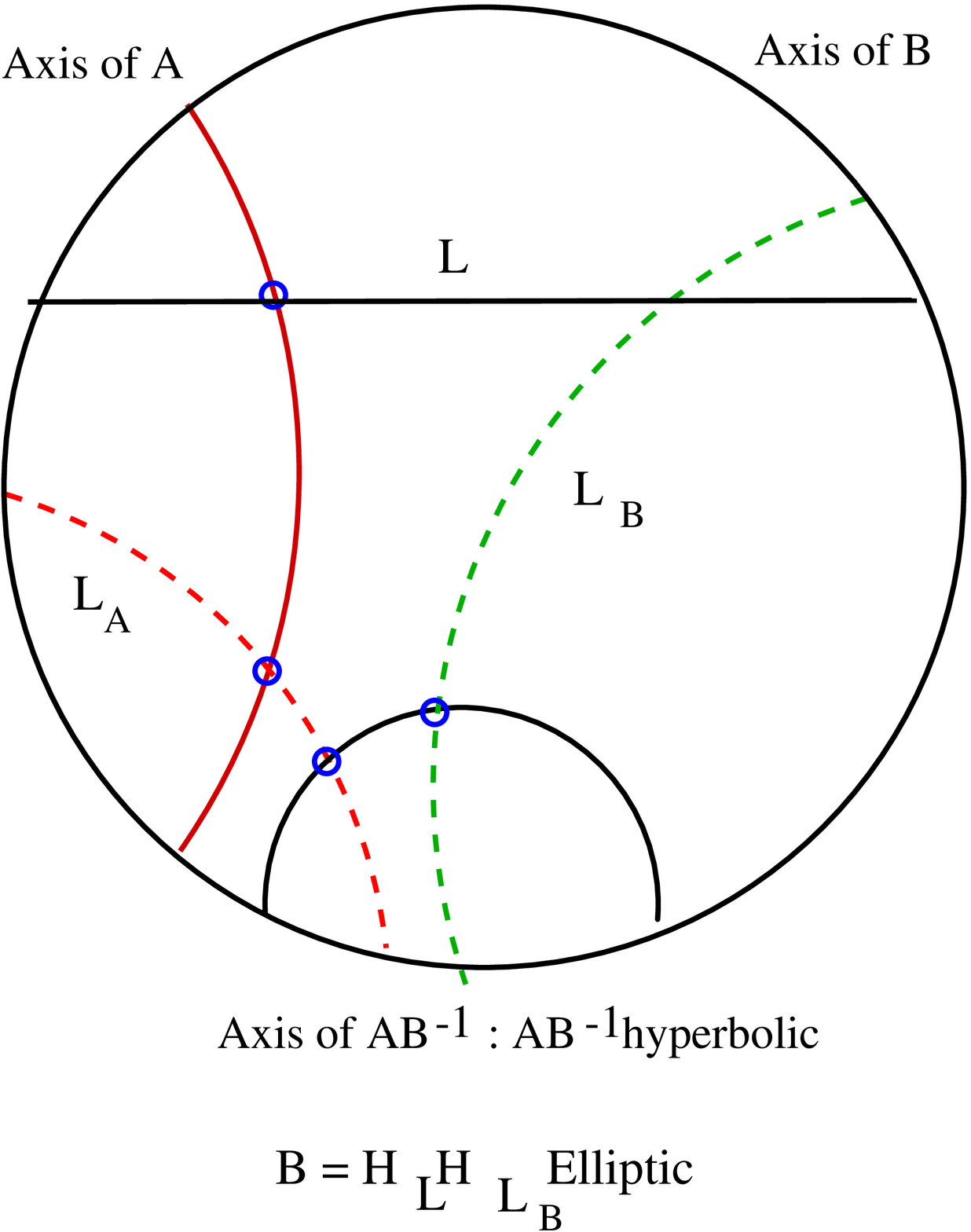}
\includegraphics[height=2in]{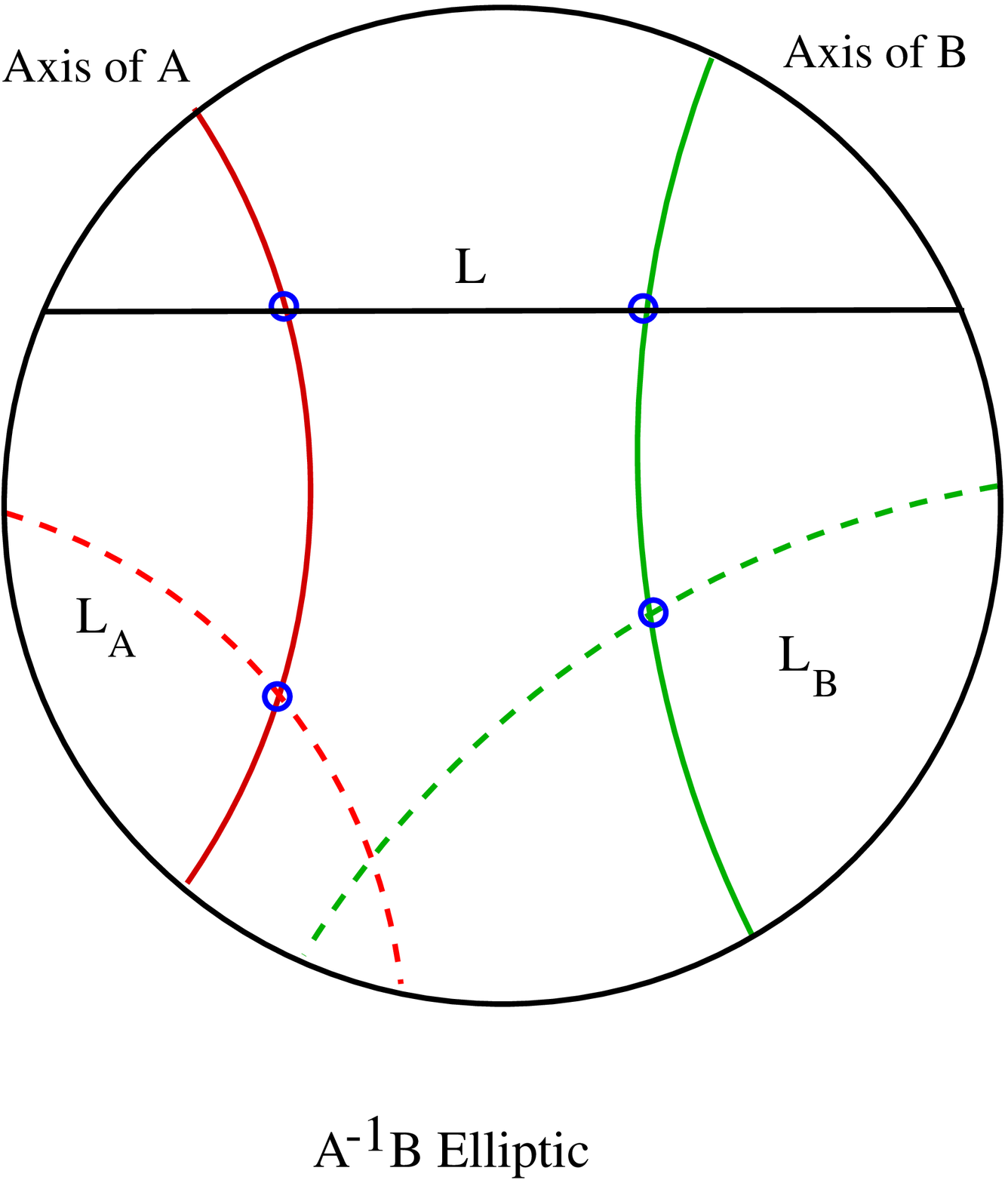} \end{center}
$\;$
\begin{center}
\includegraphics[height=2in]{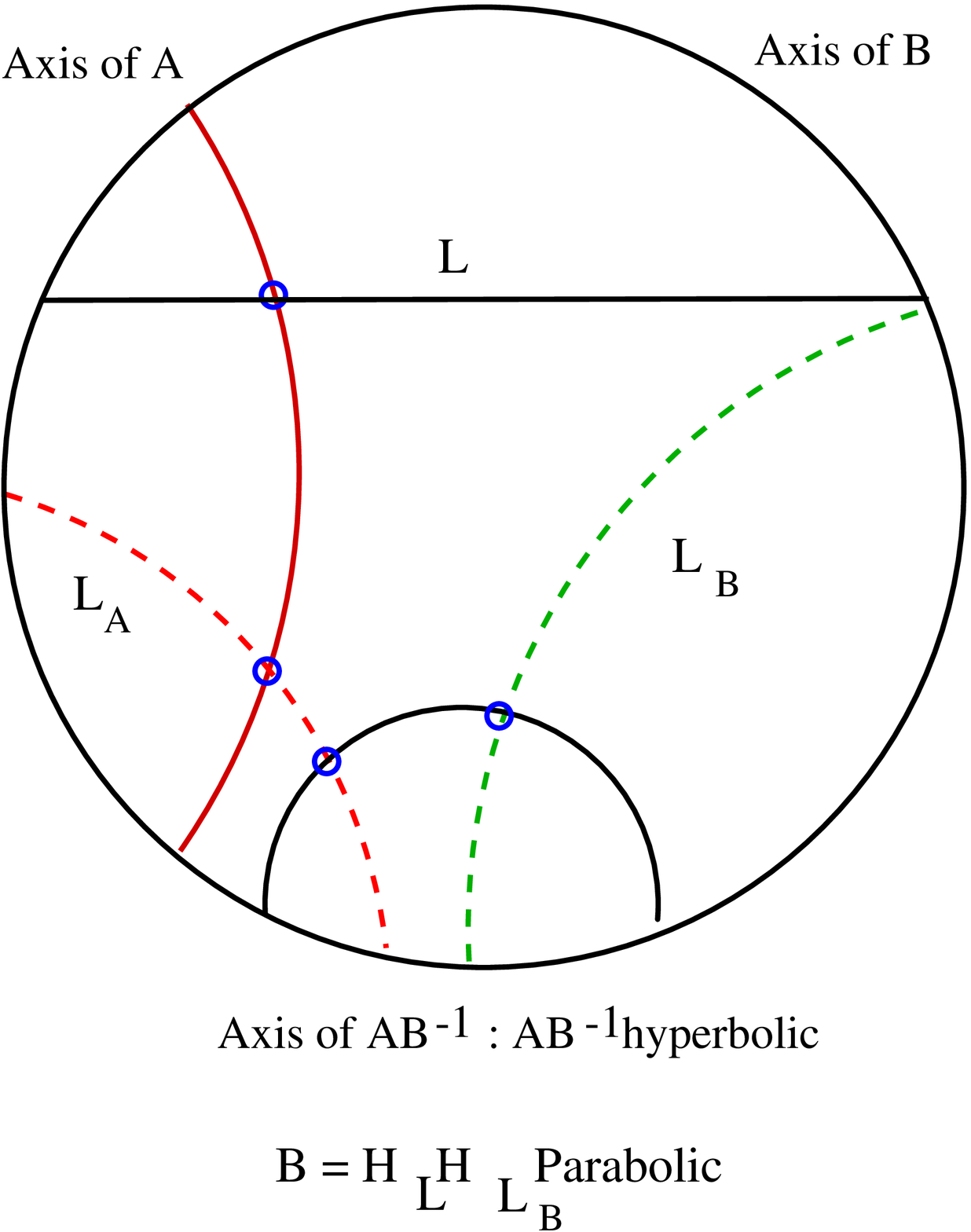}
 $\;\;$ \includegraphics[height=2in]{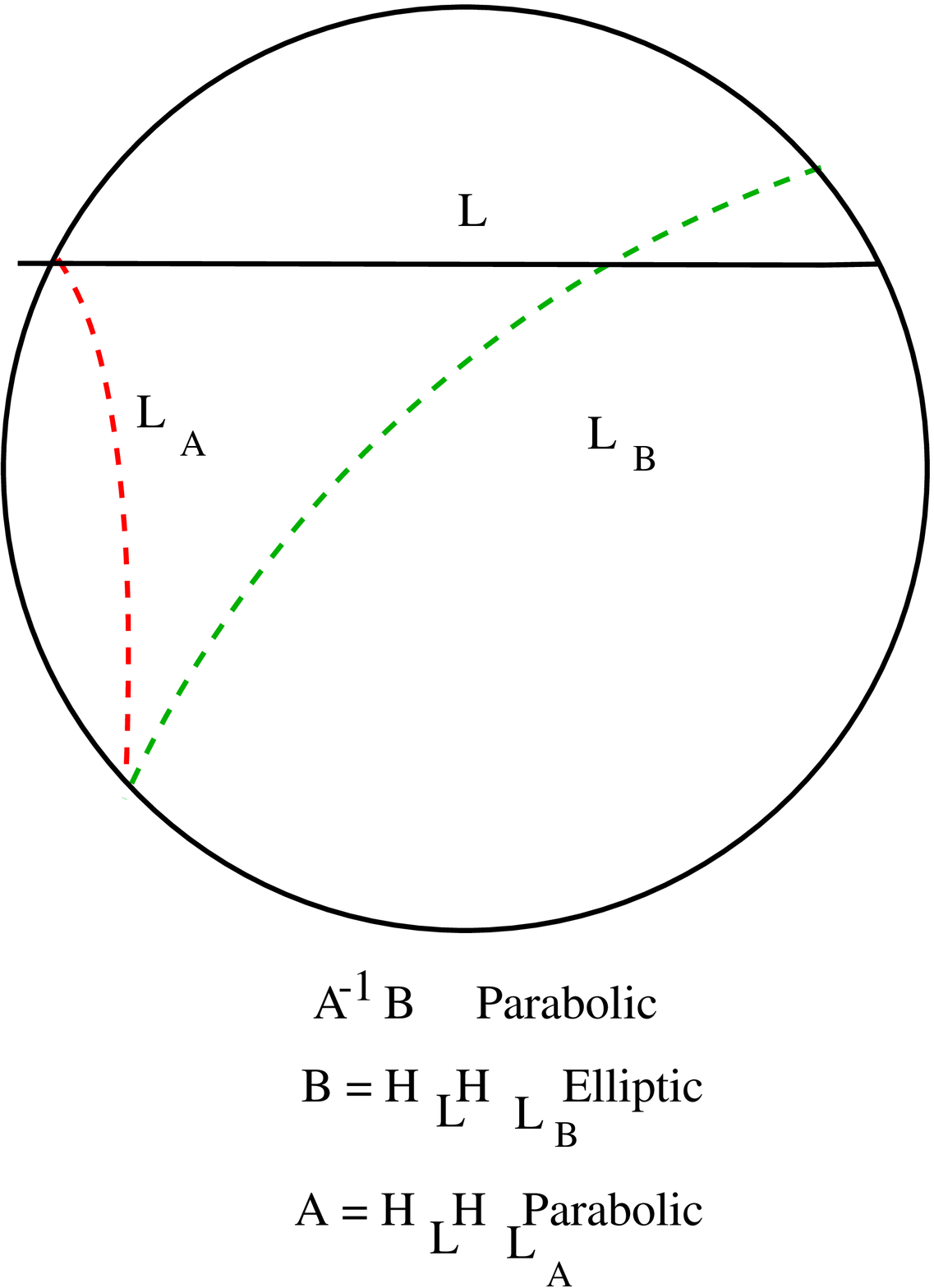}
 \end{center}

\caption{Some discrete stopping configuration $HEH, HHE,HPH,PEP$ \label{fig:someSTOP}} \end{figure}
\begin{proof}

For each pair ordered pair of generators where the order of the elements is determined by type, there are three types of subcases which are H,P, or E. We list those that are discrete stopping configurations
following \cite{GM}: 
page 16 (I-7); page 24, (II-5); page 25 (III-5), page 26 (IV-5);  page 29 Theorem;   page 30 (VI-6),  page 30 (VI-9), they are

\begin{enumerate}
\item HxH
(i) H:
 $Ax_A, L, Ax_B, L_B, Ax_{A^{-1}B},L_A $
\item HxP
(i) H: $Ax_A, L, pp_B, L_B, Ax_{A^{-1}B}, L_A$
\item PxP
(i) H: $pp_A, L, pp_B, L_B, AX_{A^{-1}B}, L_A$; (ii) P: $pp_A, L, pp_B, L_B, pp_{A^{-1}B}, L_A$
\item HxE 
(i)  H: $Ax_A, L, p_B, L_B, Ax_{A^{-1}B},L_A$
(ii) P: $Ax_A, L, p_B, L_B, pp_{A^{-1}B},L_A$
\item PxE 
(i) H: $pp_A, L, p_B, L_B, Ax_{A^{-1}B},L_A$
(ii) P: $p_A, L, p_B, L_B, pp_{A^{-1}B},L_A$
\item ExE 
(i) H:  $p_A, L, p_B, L_B, Ax_{A^{-1}B},L_A$
(ii) P: $p_A, L, p_B, L_B, pp_{A^{-1}B}, L_A$
(iii) E: $p_A, L, p_B, L_B, p_{A^{-1}B}, L_A$
\end{enumerate} \end{proof}

Because some half-turn lines reduce to points, for clarity we identify the geometry of the stopping configurations more specifically as follows in the next Theorem. The references to \ref{fig:someSTOP} and \ref{fig:Convex-Non-convex} are to be taken modulo a permutation of the order of the generators illustrated in the figures.

 \begin{thm} \label{thm:allstop}
The {\sl discrete stopping configurations}  are
\begin{enumerate}
\item \label{item:HH} {\rm HxHxH} The configuration is a convex hexagon as shown in Figure \ref{fig:Convex-Non-convex}
\item \label{item:HP} {\rm HxPxH}  The configuration is a pentagon, the convex hexagon of \ref{item:HH} where the $Ax_B$ is replaced by a point on the boundary of the unit disc where $L$ and $L_B$ meet. See Figure  \ref{fig:someSTOP}.

   \item \label{item:PP} {\rm PxP}  The configuration is a pentagon, the convex hexagon of \ref{item:HP} where the $Ax_A$  and the $Ax_B$ are replaced respectively by points $p_A$ and $p_B$ on the boundary of the unit disc where one of the following happens: $L_A$ and $L_B$ are disjoint so the figure looks like a quadrilateral or  $L_A$ and $L_B$ intersect on the boundary so the figure looks like a triangle with all each vertex on the boundary of the disc.
\item \label{item:HE} {\rm HxE}  Here  the $Ax_B$ is replaced by a point interior to the unit disc where $L$ and $L_B$ meet,  and either $L_A$ and $L_B$ are disjoint, so that $A^{-1}B$ is hyperbolic and the figure
         is a pentagon (see Figure  \ref{fig:someSTOP}) or $L_A$ and $L_B$ intersect on the boundary with $A^{-1}B$ parabolic and the figure is a quadrilateral.
    \item \label{item:PE} {\rm PxE}  The configuration is a pentagon, the convex hexagon of \ref{item:HP} where the $Ax_A$  and the $Ax_B$ are replaced respectively by points $p_A$ on the boundary and  $p_B$ interior to the unit disc with $L$ the geodesic connecting these two points and where $L$ and $L_A$ meet at $p_A$ and $L$ and $L_B$ at $p_B$.  If $L_A$ and $L_B$ are disjoint, the figure is a quadrilateral and if $L_A$ and $L_B$ intersect on the boundary or in the interior , the figure is a triangle (see Figure \ref{fig:someSTOP}).

       \item \label{item:EE} {\rm ExE} There are three cases: The hexagon reduces to

\noindent    (i) a quadrilateral  with  $L_A$ and $L_B$ disjoint
when we have {\rm ExExH}: $p_A, L, p_B, L_B, Ax_{A^{-1}B},L_A$.
\noindent (ii) a triangle with two interior vertices and one on the boundary of the disc when we have {\rm ExExP}: $p_A, L, p_B, L_B, pp_{A^{-1}B}, L_A$

\noindent (iii) a triangle with all interior vertices when we have
\noindent {\rm ExExE}:  $p_A, L, p_B, L_B, p_{A^{-1}B}, L_A$.        \end{enumerate}

\end{thm}

\begin{proof}
Follow the GM algorithm through to each discrete stopping case.
\end{proof}
\begin{rem}{\rm
In the above lists, the stopping generators are given in the order found in the GM algorithm. Later we will see that for our purposes the order does not matter.}
\end{rem}
\begin{rem}{\rm
Thus in what follows we can modify the cyclic order of the stopping generators and consider, for example, HHP and HHE instead of HPH and HEH. That is, the discrete stopping configuration can be rotated, as needed.}
\end{rem}\begin{rem}{\rm
We note that in all of these cases the convex stopping hexagons lie below (that is, to the right of) $L$ if $L$ is oriented from the axis of $A$ towards the axis of $B$ and the smaller rotation angles of elliptics and parabolics are interior to the hexagon. This assumption allows to ignore consideration of traces of pull-back to $SL(2,\RR)$ or coherent orientation used in other papers.}
\end{rem}

We begin with square roots.
 \section{Adjoining Square Roots} \label{sec:SQroots}

 We consider adjoining $B^{1/2}$,  the square root of $B$, in the cases  above where $B$ is hyperbolic. The results depends upon the location of $\B2$ as it enters and exits the hexagon. There are essentially three possibilities, but since the conclusion includes the possibilities of the new group being either free or not free, the results of the theorem are stated using more cases. In  Figures \ref{fig:TryThree1} and \ref{fig:TryThree2} we show some possible locations for $\B2$. The hyperbolic law of sines is used to position some of the geodesics.
\begin{figure}\begin{center}
(i)  \includegraphics[height=2in]{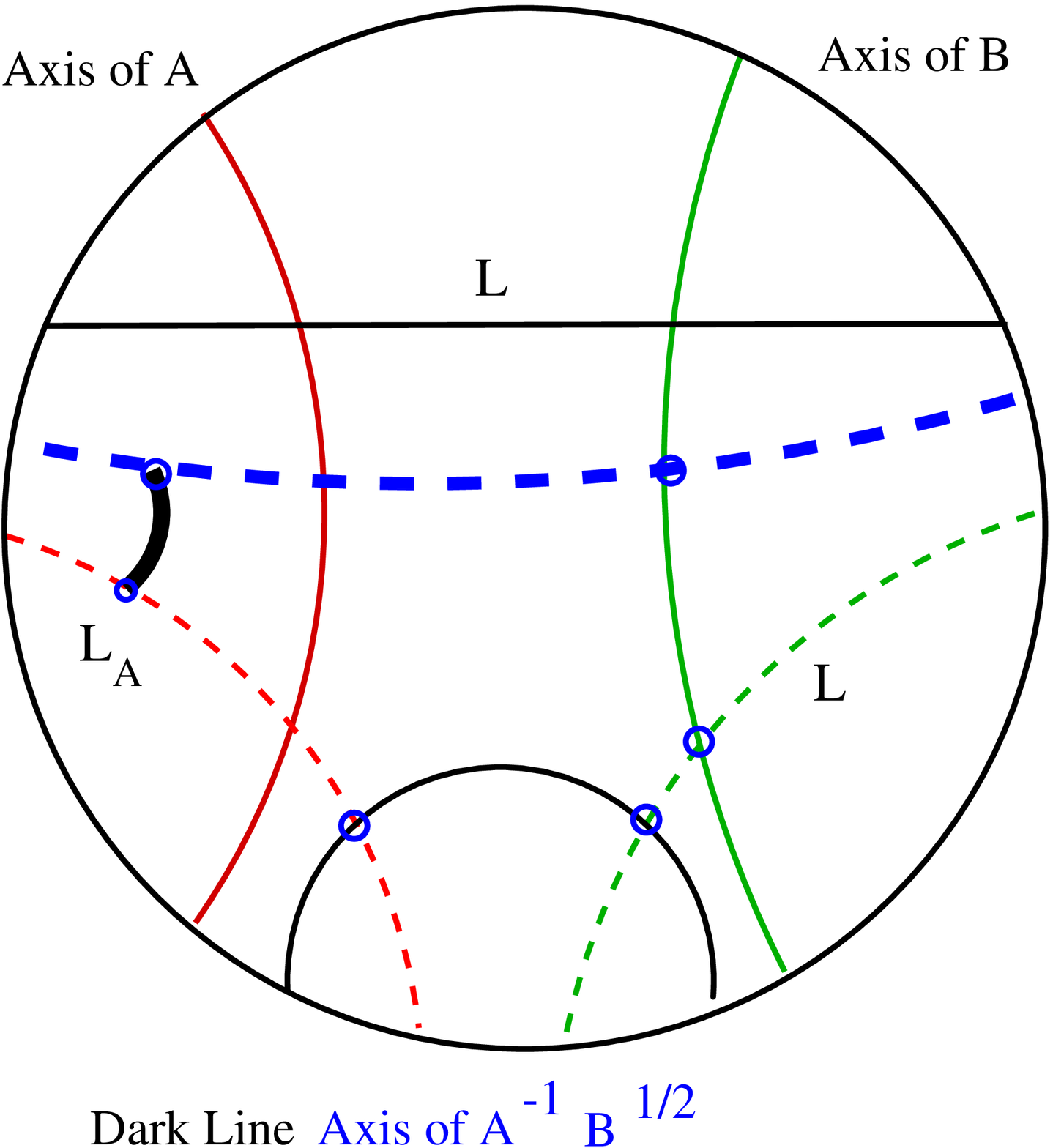}
(ii)  \includegraphics[height=2in]{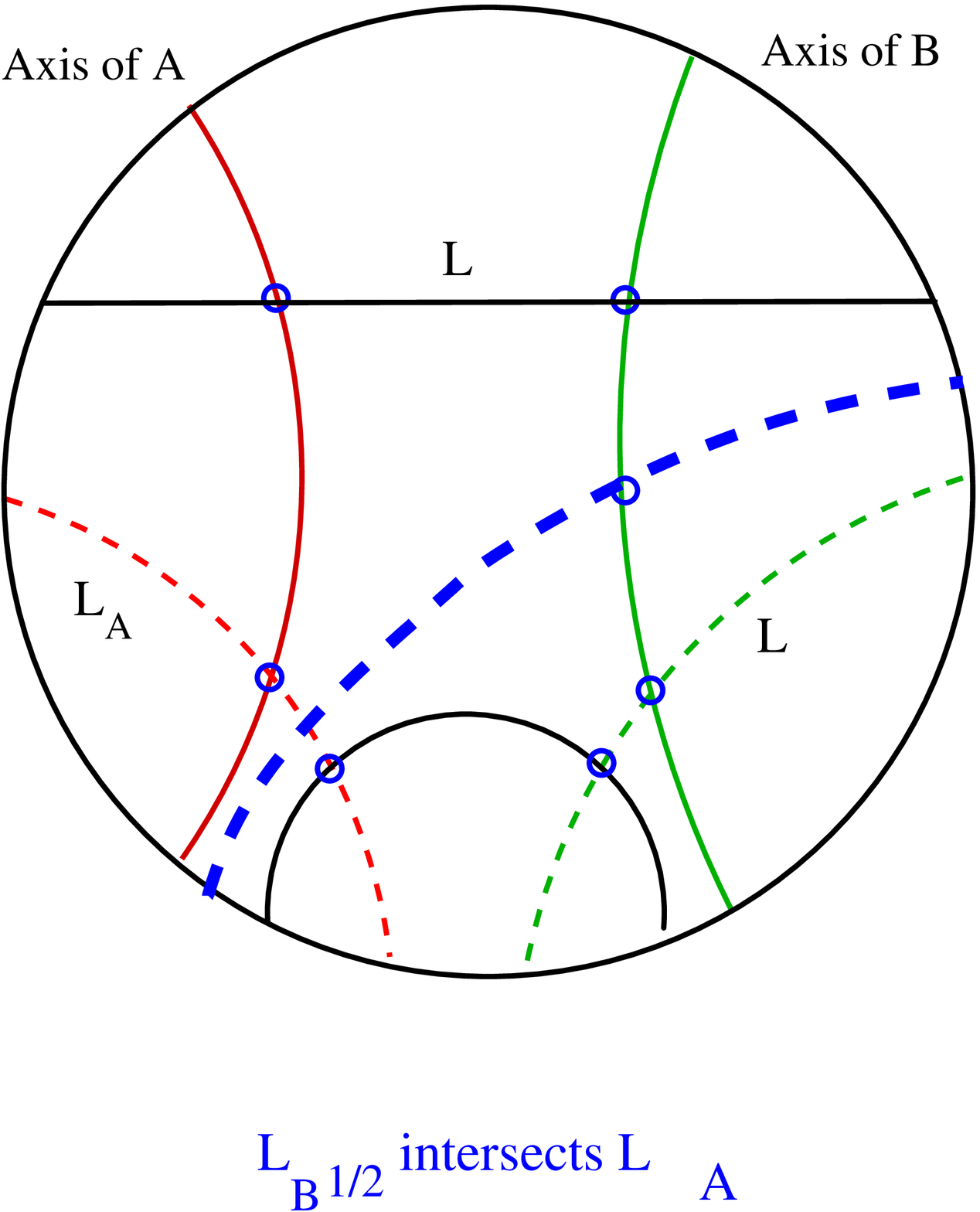}\end{center}
\caption{$A^{-1}B$ hyperbolic and $L_{B^{1/2}}$ intersects $Ax_A$, $L_A$ \label{fig:TryThree1}}
\end{figure}
\begin{figure}
\begin{center}
(iii)\includegraphics[height=2in]{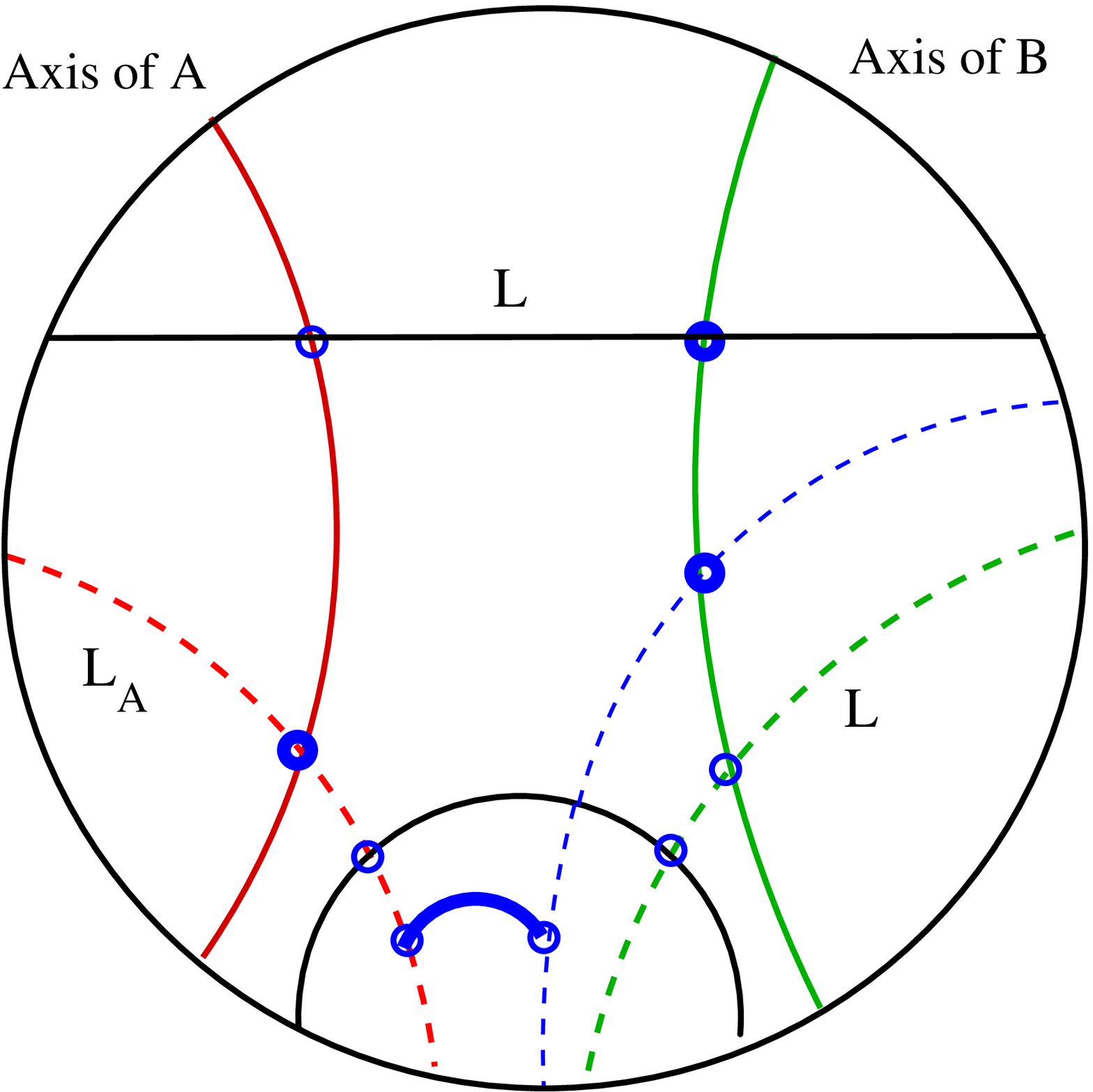}
(iv)\includegraphics[height=2in]{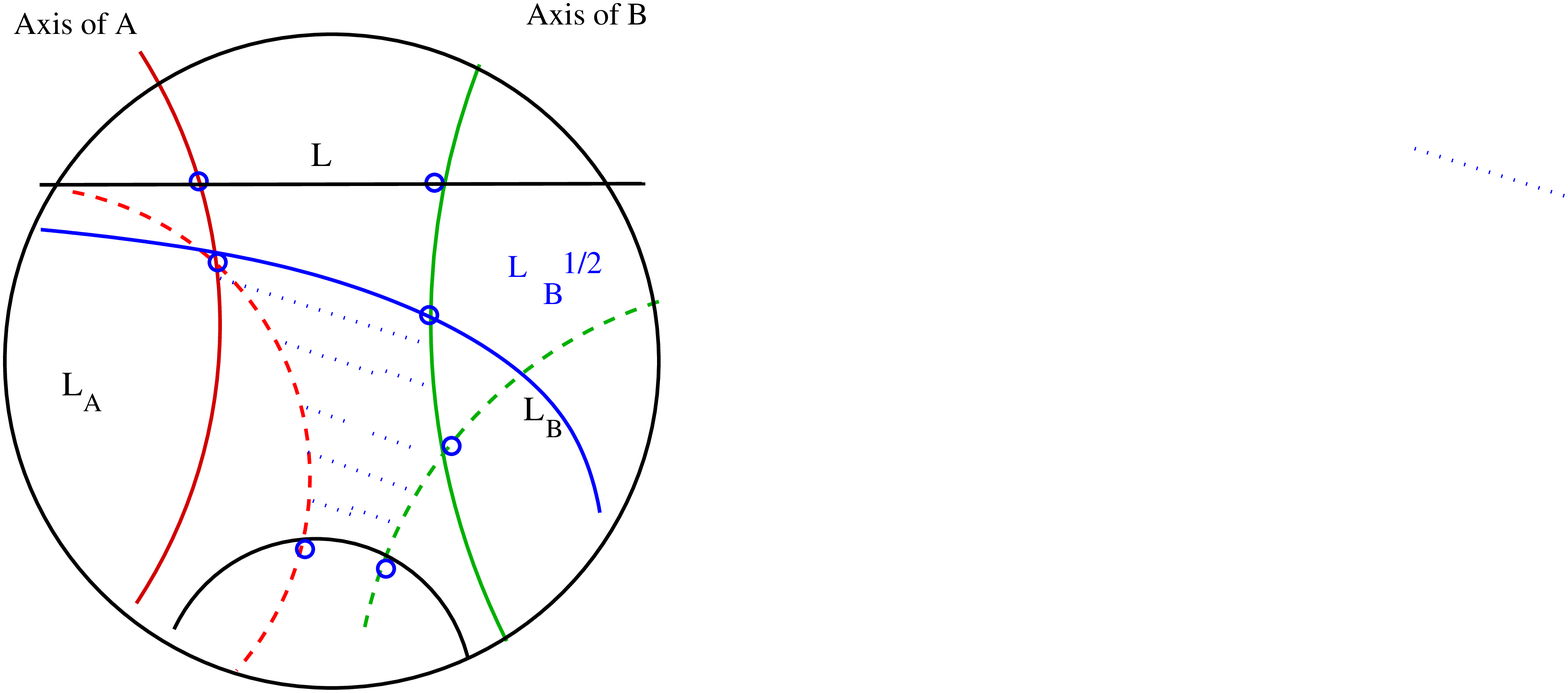}
\end{center}
\caption{$A^{-1}B$ hyperbolic and $L_{B^{1/2}}$  (iii) intersects $Ax_{A^{-1}B}$ in its interior or (iv) intersects $L_A$at a vertex \label{fig:TryThree2}}
\end{figure}

\begin{thm}\label{thm:neccsuffSQRT}
 Assume that  $(A,B)$ are a pair of hyperbolic discrete stopping generators for $G = \langle A, B \rangle$ with hexagon sides $Ax_A, L, Ax_B, L_B, Ax_{A^{-1}B}, L_A$.
 Let $\B2$ be chosen perpendicular to $Ax_B$ so that $B^{1/2} = H_L H_{\B2}$.
  \vskip .1in
  \noindent {\bf (H)} If  $A^{-1}B$ also hyperbolic,  then  $\hat{G}= \langle A, B^{1/2} \rangle$ is discrete and free  $\iff$ either

      (i) $L_{B^{1/2}} \cap Ax_A \ne \emptyset$ or
      (ii) $L_{B^{1/2}} \cap Ax_{A^{-1}B}  \ne \emptyset$

     \noindent  but  neither intersection is a vertex of the hexagon interior to $\HH^2$.
\vskip .1in
\noindent {\bf (P)} If $A^{-1}B$ is parabolic, then  $\hat{G}= \langle A, B^{1/2} \rangle$ is discrete and free $\iff$
 either

 (i) $L_{B^{1/2}} \cap Ax_A \ne \emptyset$ or
 (ii) $L_{B^{1/2}} \cap Ax_{A^{-1}B}  \ne \emptyset$

 \noindent but neither intersection is a vertex of the hexagon interior to $\HH^2$.
\vskip .1in
\noindent {\bf (E) }If $A^{-1}B$ is elliptic,it is primitive since it is a stopping generator. Then  $\hat{G}= \langle A, B^{1/2}  \rangle$ is discrete  $\iff$
 either

 (i) $L_{B^{1/2}} \cap Ax_A \ne \emptyset$ but the intersection is not at the vertex $Ax_A \cap L_A$

 or

 (ii) $L_{B^{1/2}} \cap L_A \ne \emptyset$ with $H_{L_A}H_{L_{B^{1/2}}}$ primitive elliptic


\vskip .15in
In all other cases, the group is not free and one applies the algorithm to the case where $A^{-1}B$ is elliptic to determine discreteness.
\end{thm}

\begin{proof}
 We consider the $HHH$ case and work first with the ordered hyperbolic generators $A,B,A^{-1}B$

 Since the hexagon is convex and $\B2$ is perpendicular to the Axis of $B$ intersecting it along the interior side of  the axis, it must also either intersects $Ax_A$, $L_A$ or $Ax_{A^{-1}B}$. If it intersects $L_A$, then $H_{L_A} H_{\B2}$ is elliptic so $\hat{G}$ is not free except and not discrete except possibly when the elliptic is  primitive.  If $\B2$ it intersects $Ax_A$ but not at a vertex of the hexagon, then the lower region, the region of the hexagon below $\B2$ is part of the hexagon for
  $G_2 = \langle A^{-1}B^{1/2}, B^{1/2} \rangle$. This is the hexagon with sides $Ax_{A^{-1}B^{1/2}}, L_A, Ax_{A^{-1}B}, L_B, Ax_B, \B2 $. Thus $G_2$ is discrete and free because the half-turn lines bound a region.
If $\B2$ intersects $Ax_{A^{-1}B}$, then the region of the hexagon above $\B2$ is a hexagon with half-turn lines $L, \B2, \mbox{ and } L_A$ with axis sides along $Ax_A$, $Ax_{A^{-1}B^{1/2}}, Ax_B$. Thus the group $G_2= \langle A, B^{1/2} \rangle$ is discrete and free. Of course, $G_2$ is the same group as $\hat{G}$.

The analysis of the cases for $A^{-1}B$ elliptic or parabolic are similar and thus omitted after noting that in the case that $A^{-1}B$ is primitive elliptic. \end{proof}

\section{Adjoining $n$th Roots} \label{sec:nth}

The ideas used in adjoining $n$th roots are the same as in adjoining square roots except that there are more cases to consider depending upon where $L_B^{1/n}$ intersects the hexagon and which of its powers intersect an interior side of the hexagon and which interior side, which intersect vertices and which do not intersect vertices. For ease of exposition we refer to a vertex as a side since it is a degenerate side.  Write $v_{A\cap L_A}$  and $v_{L_A\cap A^{-1}B}$ for the vertices.

For any stopping configuration we have a hexagon corresponding to transformations $A,B$ and $A^{-1}B$ and half-turn sides $L, L_A,L_B$. 
    We consider $n$th roots of $B$. The segment of $L_{B^{1/n}}$ that passes through the hexagon can exit along $L_A$, $Ax_A$ or $Ax_{A^{-1}B}$ (It cannot cross $L$ or $Ax_X$ or $L_B$. )  If it crosses $L_A$ then the group has an elliptic element and one goes to the elliptic case where $A^{-1}B$ is elliptic and then applies the algorithm appropriately.

While we know that $\Bn$ intersects $Ax_A$ in its interior, there are five choices for where each of the other  $L$-lines described below exits the hexagon: interior to $Ax_A$, interior to $L_A$, interior to $Ax_{A^{-1}B}$, at $v_{A \cap L_A}$ or at $v_{L_a \cap Ax_{A^{-1}B}}$.  The cases we need to consider  involve two integers $r$ and $s$ with $1 < s < r <n$. We assume that going counter clockwise from $L_{\Bn}$, one encounters next $L_{B^{(1/n)}s}$
and then $L_{B^{(1/n)}r}$.
We term these integers {\sl splitting integers} if they determine a jump in the side of the hexagon that  these $L-$lines  intersect. That is, if $L_{B^{(1/n)s}}$  and $L_{B^{(1/n)r}}$ do not intersect the same side.

Considering all of the possibilities gives:


 \begin{thm} \label{thm:nthroots}
 If $A,B$ are hyperbolic discrete  stopping generators
  with $A^{-1}B$ hyperbolic, parabolic or elliptic, consider the  stopping configuration  along with $\Bn $ and $L_{B^{r/n}}$ for $r$ an integer with $1 \le  r \le n$. Let $s$ be an integer with $1 \le s \le r$ so that $r$ and $s$ are splitting integers.
\vskip .1in
 {\centerline{\bf CASE I: $L_{B^{r/n}}$ does not intersect any vertex of the hexagon.}}
\vskip .15in
\noindent  {\bf H} Assume that $A^{-1}B$ also hyperbolic.
 Then  $\hat{G} = \langle A, B^{1/n} \rangle $
 is discrete and free $\iff$ either \begin{enumerate}

 \item $L_{B^{1/n}} \cap Ax_{A^{-1}B} \ne \emptyset$ or

  \item $L_{B^{(s)/n}} \cap Ax_A  \ne \emptyset$ and $L_{B^{(s+1)/n}} \cap Ax_{A^{-1}B } \ne \emptyset$, for some integer $s \le r$
      \end{enumerate}

 \noindent  If $L_{B^{s/n}} \cap L_A \ne \emptyset $ for some integer $1 \le s \le n$, then  $\langle A, B^{1/n} \rangle $ is not free, it may be discrete if $A^{-1}B^{s/n}$ is primitive, otherwise  one must go to an appropriate elliptic case of the algorithm to determine discreteness.

\vskip .15in
\noindent {\bf P} Assume that  $A^{-1}B$ is parabolic so that its axis is a boundary vertex. Then   $\hat{G} = \langle A, B^{1/n} \rangle $  is discrete and free $\iff$
   $$L_{B^{(s)/n}} \cap Ax_A  \ne \emptyset\;\;  \forall s, 1 \le s \le r.$$
If $L_{B^{s/n}} \cap L_A \ne \emptyset $ for some integer $1 \le s \le n$ but the intersection is not at the point $Ax_{A^{-1}B}$, then  $\hat{G} = \langle A, B^{1/n} \rangle $
is not free. It is discrete if $A^{-1}B^{s/n}$ is primitive. Otherwise  one must go to an appropriate elliptic case of the algorithm to determine discreteness.

\vskip .1in
{\bf E} Assume $A^{-1}B$ is elliptic so that its axis is an interior point.
Then   $\hat{G} = \langle A, B^{1/n} \rangle $
 is discrete  if
  $L_{B^{(s)/n}} \cap Ax_A  = \emptyset \;\; \forall s$ and  $L_{B^{s/n}} \cap Ax_{A^{-1}B} = \emptyset$

   or if $L_{B^{s/n}} \cap Ax_{A^{-1}B} = Ax_A \cap L_A$ for some $s$ and
    $A^{-1}B^{1/n}$ is primitive elliptic.

   For  all other cases, the elliptic case of algorithm
   must be applied to determine discreteness.

   \vskip .15in
{\centerline{\bf CASE II:  $L_{B^{r/n}}$ intersects a vertex, either $v_{L_A \cap Ax_A}$ or $v_{L_A \cap Ax_{A^{-1}B}}$}}

\vskip .03in
If either of these intersections are on the boundary of $\HH^2$, the group is discrete and free. Intersections at interior vertices will give elliptic elements and the group will be discrete if the rotation of the elliptic is primitive. Otherwise apply the GM algorithm for elliptic elements.
\end{thm}

It follows immediately that
\begin{thm}\label{thm:ratls/n}
If $s/n$ is a rational number with $s/n > 1$ and $s = wn +r$, we note that $\langle A, B^w \rangle$ is discrete whenever $\langle A, B \rangle$ is and we can apply the Theorem \ref{thm:nthroots} then to $(A,Y^{r/n})$ where $Y= B^w$.
\end{thm}


\subsection{$B$ parabolic or elliptic} \label{sec:PE}
In the case that the stopping generator $B$ is parabolic or elliptic,   $\Bn$ will have a segment that begins at $Ax_B$ which in this case is a point and passes through the interior of the stopping hexagon and the options for exiting the hexagon are unchanged. Thus we can conclude

\begin{thm} \label{thm:BPE}
If $B$ is parabolic, the conclusions of Theorem \ref{thm:nthroots} still apply, as do that those of Theorem \ref{thm:ratls/n}.

If $B$ is primitive elliptic, the conclusions of Theorem \ref{thm:nthroots} with $\hat{G}$ discrete, but not free.
The conclusions of Theorem \ref{thm:ratls/n} also apply again with $\hat{G}$ discrete, but not free.

\end{thm}

\section{General Formulation Theorems} \label{sec:allmainTHMS}

In this section we state the results above in greater generality. Assume that $(X,Y)$ are discrete stopping generators. This means that the hexagon is convex and that the angles at the any elliptic vertices are half of a primitive  elliptic angle  and the direction of rotation of parabolics and elliptics is interior to the hexagon.

The hexagon sides are $Ax_X, Ax_Y$ and $AX_{X^{-1}Y}$ and the half-turn geodesic sides as $L, L_X, \; \mbox{and} \; L_Y$. All of the half-turn sides are subintervals of proper geodesics. The axis sides may be single points.
In the case of an elliptic element its axis is a point interior to $\HH^2$ and in the case of a parabolic element its axis is a point on the boundary of  $\HH^2$.

\begin{thm} \label{thm:main} Assume that $(X,Y)$ are discrete stopping generators so that the hexagon with sides $Ax_X,L,Ax_Y, L_Y, Ax_{X^{-1}Y}, L_X$ is a convex stopping hexagon. Let $\hat{G} = \langle X^{1/n}, Y \rangle$ where $n$ is a positive integer and $X$ is any type of transformation, H, E or P.
Let $L_{X^{1/n}}$ be the geodesic with $X^{1/n} = L \circ L_{X^{1/n}}$.

\vskip .03in
There are three possibilities:
\noindent{\begin{enumerate}
\item \label{item:i} $\Xn \cap Ax_Y \ne \emptyset$
\item \label{item:ii} $\Xn \cap Ax_{X^{-1}Y} \ne \emptyset$
\item \label{item:iii}$\Xn \cap L_Y \ne \emptyset$.
\end{enumerate}}

We have
\begin{description}

\item[No vertex intersections] Assume that none of these intersections are at vertices of the hexagon, then

$$\tilde{G} = \langle X^{1/n}, Y \rangle$$ is discrete $\iff$

item \ref{item:i} or \ref{item:ii} occurs or

item \ref{item:iii} occurs with $H_{X_n}H_{L_Y}$ a primitive rotation.

\item[Vertex Intersections]

\vskip .03in

An intersection that occurs at a vertex will  be either at $Ax_Y \cap L_Y$ or $Ax_{X^{-1}Y} \cap L_Y$.

\vskip .03in

\noindent The group $\tilde{G}$ is discrete

if the vertex is interior and the rotation there is primitive.

\vskip .03in

\noindent The group $\tilde{G}$ is free

in cases \ref{item:i} or \ref{item:ii} or

if \ref{item:iii} occurs with the intersection point on the boundary of $\HH^2$.

\end{description}

\end{thm}
\begin{proof}
This follows directly from applying Theorems \ref{thm:neccsuffSQRT},
 \ref{thm:nthroots}  and \ref{thm:ratls/n}, and  \ref{thm:BPE} but  allowing the order of the generators to be cyclically permuted.
\end{proof}
We have
\begin{thm}\label{thm:ratlpowersX}

 If $X,Y$ are hyperbolic discrete  stopping generators
  with $X^{-1}Y$ hyperbolic, parabolic or elliptic, consider the  stopping configuration  along with $\Xn $ and $L_{X^{r/n}}$ for $r$ an integer with $1 \le  r \le n$. Let $s$ be an integer with $1 \le s \le r$.
\vskip .05in
\begin{description}
\item[Case I]  Assume first that $L_{X^{r/n}}$ does not intersect any vertex of the hexagon.
\vskip .15in
\begin{description}
\item[IH] Assume that $X^{-1}Y$ also hyperbolic.
\vskip .02in
 Then  $\hat{G} = \langle X^{1/n}, Y \rangle $
 is discrete and free $\iff$ either \begin{enumerate}
 \item $L_{X^{1/n}} \cap Ax_{X^{-1}Y} \ne \emptyset$ or

  \item $L_{X^{(s)/n}} \cap Ax_Y  \ne \emptyset$ and $L_{X^{(s+1)/n}} \cap Ax_{X^{-1}Y } \ne \emptyset$, for some integer $s \le r$.
      \end{enumerate}

\vskip .02in

   If $L_{X^{s/n}} \cap L_Y \ne \emptyset $ for some integer $1 \le s \le n$,

      then  $\langle X^{-1}, Y \rangle $ is not free,

      it may be discrete if $X^{s/n}Y^{-1}$ is primitive,

      otherwise  one must go to an appropriate elliptic case of the algorithm to determine discreteness.

\vskip .15in
\item[IP] Assume that  $X^{-1}Y$ is parabolic so that its axis is a boundary vertex.

Then   $\hat{G} = \langle X^{1/n}, Y \rangle $  is discrete and free $\iff$
   $$L_{X^{(s)/n}} \cap Ax_Y  \ne \emptyset\;\;  \forall s, 1 \le s \le r.$$

\vskip .031in
If $L_{X^{s/n}} \cap L_Y \ne \emptyset $ for some integer $1 \le s \le n$ but the intersection is not at the point $Ax_{X^{-1}Y}$,

then  $\hat{G} = \langle X^{1/n}, Y \rangle $
          is not free.

          It is discrete if $X^{s/n}Y^{-1}$ is primitive.

          Otherwise  one must go to an appropriate elliptic case of the algorithm to determine discreteness.


\vskip .1in
\item[IE] Assume $X^{-1}Y$ is elliptic so that its axis is an interior point.

Then   $\hat{G} = \langle X^{1/n}, Y \rangle $
 is discrete  if

  $L_{X^{(s)/n}} \cap Ax_Y  \ne  \emptyset \;\; \forall s$ but $L_{X^{s/n}} \cap L_Y =  \emptyset$
   unless the intersection is at $Ax_Y \cap L_Y$ and
    $H_{L_{X^{s/n}}}H_{L_Y}$ is primitive elliptic.
\vskip .02in

    For  all other cases, the elliptic case of algorithm
   must be applied to determine discreteness.

   \vskip .05in 
   \end{description}
\item[II] If $L_{X^{r/n}}$ intersects a vertex, it must be at $L_Y \cap Ax_Y$ or $L_Y \cap Ax_{X^{-1}Y}$. If either of these intersections are on the boundary of $\HH^2$, the group is discrete and free. Intersections at interior vertices will give elliptic elements and the group will be discrete if the rotation of the elliptic is primitive. Otherwise apply the elliptic cases of the algorithm.
\vskip .05in
\item[Case III]  If $X$ is parabolic, the conclusion of \ref{thm:main} still apply.
\vskip .05in
\item[Case IV] If $X$ is primitive elliptic, the conclusion of \ref{thm:main} with $\hat{G}$ discrete but not free.
    \end{description}
\end{thm}
\begin{proof}
Applying Theorems
 \ref{thm:nthroots}, \ref{thm:BPE} and \ref{thm:main} to the permutation $(X,Y,X^{-1}Y)$ of the triple $(X,Y,X^{-1}Y)$.
\end{proof}

If $s$,   $n$, $w$ and $r$ are positive integers with $s/n > 1$ and $S = wn +r,$  then  $G = \langle X^w,Y \rangle$ is discrete whenever $\langle X, Y \rangle$ is.  Applying the above we have immediately
\begin{thm} \label{thm:moregen} {\rm [Powers of Roots]}
Let  $s$,   $n$, $w$ and $r$ be positive integers with $s/n > 1$ and $s = wn +r$. The group $  \langle X^w,Y \rangle$ is discrete whenever $\langle Z, Y \rangle$ is where  $Z=X^{r/n}$ and the discreteness of $\langle Z, Y \rangle$ can be determined by Theorem \ref{thm:ratlpowersX}
\end{thm}

\section{Miscellaneous Remarks}
 \label{sec:nonstop}
 \begin{rem} {\bf Roots of Non-stopping generators}
{\rm A generator $X$ for a rank two discrete free group $G= \langle A, B \rangle$ is a {\sl primitive generator} if there exists an element $Y$ such that $G= \langle X, Y \rangle$
\cite{MKS}.  The pair $(X,Y)$ is a called a 
{\sl a primitive pair}.

Given a primitive pair, if $G$ is discrete and free, there is a sequence of integers, known as  the F-sequence or the Fibonacci sequence, $[n_1,...,n_t]$ such that the sequence   stops at a pair $(C,D)$ of discrete stopping generators after applying appropriate Nielsen transformations determined by the $n_i$  starting with the pair $(X,Y)$.
Using the reverse $F$ sequence, one can write $X$ and $Y$ as  words in the stopping generators and thus obtain $(X,Y)$ as words in $(C,D)$. One can apply the GM algorithm to $(X^{r/n},Y)$ to see whether the group is discrete or not.
Starting with $X$ and $Y$ written as words in $C$ and $D$ will often shorten the implementation of the algorithm
Alternately, if it is known that $(X,Y)$ is discrete and free, one can apply the GM algorithm to find its stopping generators $(C,D)$ and then write $(X,Y^{r/n})$ as words in $(C,D)$ before running the algorithm.}
\end{rem}

\begin{rem}{\rm
If $G$ contains elliptic elements, there is an {\sl extended $F$-sequence} \cite{Vidur}. It contains extra terms that correspond geometrically to replacing an elliptic element by its primitive power and the extra integer is that power.
The same idea applies.}
\end{rem}

\begin{rem}{\bf Matrix calculations} {\rm
Using Fenchel's theory of matrices and extending it as  necessary allows one to turn these geometric algorithms into purely computational matrix procedures. We use, for example, we some of the following results  from \cite{Fench}.
(i) If ${\bf f} \in SL(2,\CC)$ is a matrix determining a transformation $f$, then ${\bf f} - {\bf f}^{-1}$ is a {\sl line matrix}. It corresponds to a half-turn about a geodesic whose ends are the fixed points of the line matrix. The geodesic is, of course, the axis of $f$. (ii) If $f$ and $g$ are transformations with distinct axes and with line matrices $L_f$ and $L_g$. The axes of $f$ and $g$ are perpendicular if the trace of $L_gL_f =0$. This holds even if $Ax_f$ and or $Ax_g$ are improper lines.
(iii) The trace of $fg$ tells us the angle of intersection (see also
\cite{Beard}).
The computational matrix theory is developed in full detail in  \cite{JGlecturenotes}.}
\end{rem}
\begin{rem} {\rm The question has been raised as to whether this translates to an algebraic treatment using the Purtzitsky-Rosenberger trace minimizing algorithm \cite{P,PR,R}. The trace minimizing method is to replace $(A,\Bn)$ when $Tr \; A \ge Tr \; \Bn$  by one of the ordered pairs $(\Bn,A\Bn)$, $(\Bn,A\Bn^{-1})$, $(A\Bn, \Bn)$ or $(A\Bn^{-1},\Bn)$ depending upon the sizes of the traces.
Thus it seems that one would have to start the algorithm with $\langle A, \Bn \rangle$ and that even if the pair $(A,B)$ were the algebraic stopping generators, computations would have to be carried out to reflect the intersection properties or the algorithmic steps.}
\end{rem}
\section*{Acknowledgement}

The author  thanks John Parker for helpful and insightful comments during the preparation of this manuscript.

\end{document}